\documentclass[11pt,twoside]{article}
\usepackage{amssymb}
\usepackage{amssymb,amsmath,amsthm,amsfonts,mathrsfs,hyperref}
\usepackage{times}
\usepackage{enumerate}
\usepackage{cite,titletoc}

\allowdisplaybreaks

\def\N{\mathop{\mathbb N\kern 0pt}\nolimits}
\def\Q{\mathop{\mathbb Q\kern 0pt}\nolimits}
\def\R{\mathop{\mathbb R\kern 0pt}\nolimits}
\def\SS{\mathop{\mathbb S\kern 0pt}\nolimits}

\hypersetup{colorlinks=true,linkcolor=blue,citecolor=red,urlcolor=cyan}

\theoremstyle{plain}
\newtheorem{theorem}{Theorem}[section]

\newtheorem{lemma}[theorem]{Lemma}

\theoremstyle{definition}

\numberwithin{equation}{section}

\title{Weak factorizations of the Hardy space in terms of multilinear fractional integral operator}

\author{Dinghuai Wang \footnote{Dinghuai Wang(\texttt{Wangdh1990$@$126.com}) is supported by
    National Natural Science Foundation of China(No.11971237,12071223), the Natural Science Foundation of the Jiangsu Higher Education Institutions of China(No.19KJA320001) and Doctoral Scientific Research Foundation.} , Rongxiang Zhu \\
   [12pt] {\small Wangdh1990@126.com; ZRx1268@163.com}\\
   [12pt] {\small School of Mathematics and Statistics, Anhui Normal University, Wuhu, 241002, China}\\}

\begin{document}

\date{}
\maketitle
\thispagestyle{empty}

\begin{abstract}
We give a constructive proof of the factorization theorem for the classical Hardy space in terms of fractional integral operator. Moreover, the result is extended to the multilinear case and weighted case. As an application, we obtain the characterization of $BMO$ via the weighted boundedness of commutators of the multilinear fractional integral operator, without individual conditions on the weights class.

\vskip 0.2 true cm


\vskip 0.2 true cm
\noindent
\end{abstract}

\vskip 0.6 true cm

\section{Introduction and Statement of Main Results}

The theory of Hardy spaces is vast and complicated, it has been systematically developed and plays an important role in harmonic analysis and PDEs. A well known result of Coifman-Rochberg-Weiss \cite{CRW1976} provided a constructive proof of the weak factorizations of the classical Hardy space $H^1$ in terms of Riesz transforms. The result depends upon the duality between $H^{1}$ and $BMO$ and upon a new result linking $BMO$ and the $L^{p}$ boundedness of certain commutator operators. Later on, Li and Wick \cite{LW2018} obtained the same results in the multilinear setting.

An interesting question arises: can we replace the Riesz transforms with other operators? In this paper, we shall show the factorization theorem for Hardy space via the multilinear fractional integral operator on weighted Lebesgue spaces. The key point of the present paper is first: to obtain the factorization theorem for the classical Hardy space in terms of fractional integral operator, it needs some tedious calculations in applications; and second: to establish the factorization theorem on weighted Lebesgue spaces without individual conditions on the weights class.

As we will work in the weighted setting, we need the notion of weighted $L^p$ space: $L^p(w)=L^p(\mathbb{R}^n,\omega dx)$ denotes the collection of
measurable functions $f$ on $\mathbb{R}^n$ such that
	$$\left \| f \right \|_{L^p(\omega)}:=\left (\int_{\mathbb{R}^n}|f(x)|^p \omega(x)dx\right)^{\frac{1}{p}}<\infty.$$
We recall the definition of $A_{p}$ weight introduced by Muckenhoupt in \cite{M1972}, which give the characterization of all weights $\omega(x)$ such
that the Hardy-Littlewood maximal operator
$$
M(f)(x)=\sup_{Q\ni x}\frac{1}{|Q|}\int_{Q}|f(y)|dy
$$
is bounded on $L^{p}(\omega)$. For $1< p<\infty$ and a nonnegative locally integrable function $\omega$ on $\mathbb{R}^n$, $\omega$ is in the
Muckenhoupt $A_{p}$ class if it satisfies the condition
$$[\omega]_{A_{p}}:=\sup_{Q}\bigg(\frac{1}{|Q|}\int_{Q}\omega(x)dx\bigg)\bigg(\frac{1}{|Q|}\int_{Q}\omega(x)^{-\frac{1}{p-1}}dx\bigg)^{p-1}<\infty.$$
We write $A_{\infty}=\bigcup_{1\leq p<\infty}A_{p}$. For $\omega\in A_{\infty}$, there exists $0<\epsilon,L<\infty$ such that for all measurable subsets $S$ of cube $Q$,
\begin{align}\label{weight-eq1}
  \frac{\omega(S)}{\omega(Q)}\leq C\Big(\frac{|S|}{|Q|}\Big)^{\epsilon}
\end{align}
and
\begin{align}\label{weight-eq2}
\left( \frac{|S|}{|Q|}\right)^{L}\le C\frac{\omega(S)}{\omega(Q)}.
\end{align}

In the celebrated work \cite{LOPTT2009} Lerner et al. established a theory of weights adapted to the multilinear setting and resolved the problems proposed in \cite{GT22002}.
For $1<p_1,\dots,p_m<\infty$, $\vec{P}=(p_1,p_2,\dots,p_m)$, and
	$p$ such that $\frac{1}{p_1}+\dots +\frac{1}{p_m}=\frac{1}{p}$, a vector weight $ \vec{\omega} =(\omega_1,\omega_2,\dots, \omega_m)$ belongs to $A_{\vec{P}}$ if
	\begin{align*}
		[\vec{\omega}]_{A_{\vec{P}}}=\sup_{Q}\left (\frac{1}{|Q|}\int_{Q}\prod_{i=1}^{m}\omega_i(x)^{\frac{p}{p_i}}dx\right) \prod_{i=1}^{m} \left(\frac{1}{|Q|}\int_{Q}\omega_{i}(x)^{1-p_i^\prime}dx\right)^{\frac{p}{p_i^{\prime}}}< \infty.
	\end{align*}
	For brevity, we will often use the notation $\nu_{\vec{\omega}}=\prod_{i=1}^{m} \omega_{i}^{\frac{p}{p_i}}$ in the first integral. For $1<p_1,\dots,p_m,q<\infty$ and $\vec{P}=(p_1,p_2,\dots,p_m)$, a vector weight $\vec{\omega}=(\omega_1,\omega_2,\dots,\omega_m)$ belongs to $A_{\vec{P},q}$ if
	\begin{align*}
	[\vec{\omega}]_{A_{\vec{P},q}}=\sup_{Q}\left (\frac{1}{|Q|}\int_{Q}\prod_{i=1}^{m}\omega_i(x)^{q}dx\right) \prod_{i=1}^{m} \left(\frac{1}{|Q|}\int_{Q}\omega_{i}(x)^{-p_i^\prime}dx\right)^{\frac{q}{p_i^{\prime}}}< \infty.
	\end{align*}
	As with the $A_{\vec{P}}$ weights, for brevity we will use $\mu_{\vec{\omega}}=\prod_{i=1}^{m}\omega_{i}^{q} $.
It was shown by Moen in \cite{M2009} that if $\vec{\omega}\in A_{\vec{P},q}$, then $\omega_{i}^{-p_i^{\prime}}\in A_{mp_i^{\prime}}$ and $\mu_{\vec{\omega}}\in A_{mq}$.
	
We now recall the definition of multilinear fractional integral operator. For $m\in \mathbb{N}$ and $0<\alpha<mn$, the multilinear fractional integral operator is defined by
\begin{align*}
	I_{\alpha}(f_{1},f_{2},\dots f_{m})(x)=\int_{\mathbb{R}^{mn}} \frac{\prod_{j=1}^{m}f_{j}(y_{j})}{(|x-y_1|+\dots+|x-y_m|)^{mn-\alpha}}dy_{1}\dots y_{m}.
 \end{align*}
For $l=1,2,\dots ,m,$ we define the multilinear "multiplication"operators $ {\textstyle \prod _{l}}$ as follows.
	\begin{align}\label{commutator}
		{\textstyle \prod _{l}} (g,h_{1},\dots,h_{m})(x):=h_{l}(I_{\alpha}^{*})_{l}(h_{1},\dots,h_{l-1},g,h_{l+1},\dots,h_{m)}-gI_{\alpha}(h_{1},\dots,h_{m})(x).
	\end{align}
where $(I_{\alpha}^{*})_{l}$ is the $l$-th partial adjoint of $I_{\alpha}$, and it is easy to see that $(I_{\alpha}^{*})_{l}=I_{\alpha}$.

Our main result is then the following factorization result for $H^{1}(\mathbb{R}^{n})$ in terms of the multilinear operator $ {\prod_{l}}$. The result is new even in the linear case and unweighted case.
\begin{theorem}\label{main1}	
Let $1\le l\le m$, $0<\alpha<mn$, $1< p_{1},\dots,p_{m},q<\infty$, $\frac{1}{p_1}+\dots+\frac{1}{p_m} -\frac{1}{q}=\frac{\alpha}{n}$ and $\vec{\omega}\in A_{\vec{P},q}$.
Then for any $f\in H^{1}(\mathbb{R}^{n})$, there exists sequences $\left \{ \lambda_{s}^{k}  \right \}\in \ell_{1}$ and functions $g_{s}^{k}\in L^{q^{\prime}}(\mu^{1-q^{\prime}}_{\vec{\omega}}),h_{s,1}^{k}\in L^{p_{1}}(\omega_1^{p_1}),\dots ,h_{s,m}^{k}\in L^{p_{m}}(\omega_m^{p_m})$ such that
    \begin{align}\label{f=}
     f=\sum_{k=1}^{\infty}\sum_{s=1}^{\infty}\lambda_{s}^{k}
     {\textstyle \prod_{l}}(g_{s}^{k},h_{s,1}^{k},\dots ,h_{s,m}^{k})
     \end{align}
     in the sense of $H^{1}(\mathbb{R}^{n})$. Moreover,
      \begin{align*}
      \|f\|_{H^{1}\left(\mathbb{R}^{n}\right)}\approx
      \inf\left\{\sum_{k=1}^{\infty}\sum_{s=1}^{\infty}\left|\lambda_{s}^{k}\right|
      \left\|g_{s}^{k}\right\|_{L^{q^{\prime}}(\mu^{1-q^{\prime}}_{\vec{\omega}})}
      \left\|h_{s,1}^{k}\right\|_{L^{p_{1}}(\omega_1^{p_1})}\cdots\left\|h_{s,m}^{k}\right\|_{L^{p_{m}}(\omega_m^{p_m})}\right\},
      \end{align*}
where the infimum above is taken over all possible representations of $f$ that satisfy \eqref{f=}.
\end{theorem}

As a direct application, we will give the characterization of $BMO$ via commutators of the multilinear fractional integral operator on weighted Lebesgue spaces, without individual conditions on the weights class.
In analogy with the linear case, we define the $l$-th possible multilinear commutators of the multilinear fractional integral operator $I_{\alpha}$ as follows.
	\begin{align*}
	\left [b,I_{\alpha} \right ]_{l}(f_{1},\dots,f_{m})(x)=I_{\alpha}(f_{1},\dots,bf_{l},\dots,f_{m})(x)-bI_{\alpha}(f_{1},\dots,f_{m})(x).
	\end{align*}
It was first shown in \cite{CX2010} that given $0<\alpha<mn$, $1<p_1,\cdots,p_m<\infty$, $\frac{1}{p}=\frac{1}{p_1}+\cdots +\frac{1}{p_m}$, and $\frac{1}{p}-\frac{1}{q}=\frac{\alpha}{n}$,
 if $(\omega_1^r,\omega_2^r,\dots,\omega_m^r)\in A_{\vec{P}/r,q/r}$ for some $r>1$ with $0<r\alpha<mn$, and $\mu_{\vec{\omega}}\in A_{\infty}$, then
   \begin{align*}
	[b,I_{\alpha}]_l:L^{p_1}(\omega_1^{p_1})\times \cdots \times L^{p_m}(\omega_m^{p_m})\to L^{q}(\mu_{\vec{\omega}})
   \end{align*}
Later on, in \cite{CW2013}, the result was improved and the explicitly stated bump condition involving
$r > 1$ was removed. Conversely, Guo, Lian and Wu \cite{GLW2020} proved that the $BMO$ is necessary for the boundedness of $[b,I_{\alpha}]$ when $\omega^{p_{i}}_{i}\in A_{\infty}$ with $i=1,\cdots,m$.
Thus, it is a nature problem to obtain the weighted results without the restriction of that $\omega^{p_{i}}_{i}\in A_{\infty}$.

\begin{theorem} \label{main2}
 Let $1\le l\le m$, $0<\alpha<mn$, $1< p_{1},\dots,p_{m},q<\infty$, $\frac{1}{p_1}+\dots+\frac{1}{p_m} -\frac{1}{q}=\frac{\alpha}{n}$ and $\vec{\omega}\in A_{\vec{P},q}$.
The commutator $\left [b,I_{\alpha} \right]_{l}$ is bounded from $L^{p_1}(\omega_1^{p_1})\times \dots \times L^{p_m}(\omega_m^{p_m})$ to $L^{q}(\mu_{\vec{\omega}})$ if and only if $b\in BMO$.
\end{theorem}

\section{Auxiliary lemmas and Proofs of Theorem \ref{main1}-\ref{main2}}

We first introduce the atomic decomposition of Hardy spaces.
Let $0<p\leq 1<q\leq \infty$. A function $a$ is called a $(p,q,s)$ atom for $H^p(\mathbb{R}^n)$ if there exists a cube $Q$ such that
\begin{itemize}
  \item [\rm (i)]  $a$ is supported in $Q$;
  \item [\rm (ii)] $\|a\|_{L^{q}(\mathbb{R}^n)}\leq |Q|^{\frac{1}{q}-\frac{1}{p}}$;
  \item [\rm (iii)] $\int_{\mathbb{R}^n} a(x)x^{\gamma}dx=0$ for all multi-indices $\gamma$ with $|\gamma|\leq [\frac{n}{p}-n]$.
  \end{itemize}

To prove Theorem \ref{main1}, we need the following auxiliary lemmas.
\begin{lemma} \label{lem-main1}
     Let $f$ be a function satisfying the following estimates:
\begin{itemize}
  \item [\rm (i)]  $\int_{\mathbb{R}^n}f(x)dx=0;$
  \item [\rm (ii)] there exist balls $B_1=B(x_1,r)$ and $B_2=B(x_2,r)$ for some $x_1,x_2\in \mathbb{R}^n$ and $r>0$ such that
     $$|f(x)|\le h_1(x)\chi_{B_1}(x)+h_2(x)\chi_{B_2}(x),$$
     where $\left\|h_i\right\|_{L^q(\mathbb{R}^n)}\le Cr^{-n/q'}$ and $1<q\leq \infty$;
  \item [\rm (iii)] $|x_1-x_2|\ge 4r$.
  \end{itemize}
  Then there exists a positive constant $C$ independent of $x_1,x_2,r$ such that
     $$\left \|f\right\|_{H^1(\mathbb{R}^n)}\le{C}\log\frac{|x_1-x_2|}{r}.$$
      \end{lemma}
\begin{proof}
Assume that  $f:=f_1+f_2$, where $|f_{i}|\leq h_{i}$ and supp $f_{i}\subset B_i$ for $i=1,2.$ We will show that $f$ has the following atomic decomposition
    \begin{align}\label{lem1-1}
    f=\sum_{i=1}^{2}\sum_{j=1}^{J_0+1}\lambda_i^ja_i^j,
    \end{align}
    where $J_0$ is the smallest integer larger than $\log \frac{|x_{1}-x_{2}|}{r}$ and for each $j$, $a_i^j$
    is a atom and $\lambda_i^j$ a real number satisfying that
    \begin{align}\label{lem1-2}
    |\lambda_i^j|\lesssim 1.
    \end{align}

 To this end, for $i=1,2$, we write
   $$f_i(x)=\left[f_i(x)-\tilde{\lambda}_i^1\chi_{B_i}\right]+\tilde{\lambda}_i^1\chi_{B_{i}}=:f_i^1(x)+\tilde{\lambda}_i^1\chi_{B_{i}},$$
   where
   $$\tilde{\lambda}_i^1:=\frac{1}{|B_{i}|}\int_{B_{i}}f_i(x)dx.$$
   Let $\lambda_i^1:=\|f^1_i\|_{L^q(\mathbb{R}^n)}|B_{i}|^{\frac{n}{q^{\prime}}}\lesssim 1$ and $a_i^1:=f_i^1/\lambda_i^1.$ From the fact that
   $$\|a_{i}^1\|_{L^{q}(\mathbb{R}^n)}=\frac{\|f_{i}^1\|_{L^{q}(\mathbb{R}^n)}}{\lambda_{i}^1}\leq |B_{i}|^{1/q-1} ,$$
   we know that $a_i^1$ is a $(1,q,0)$-atom supported on $B_{i}$ and $\lambda_i^1$ satisfies \eqref{lem1-2}.
We further write $$\tilde{\lambda}_i^1\chi_{B_{i}}=\tilde{\lambda}_i^1\chi_{B_{i}}-\tilde{\lambda}_i^2\chi_{2B_i}+\tilde{\lambda}_{i}^2\chi_{2B_i}
   =:f_i^2+\tilde{\lambda}_{i}^2\chi_{2B_i},$$
   where $$\tilde{\lambda}_{i}^2:=\frac{1}{|2B_i|}\int_{2B_i}f_i(x)dx.$$
   Let
   $\lambda_i^2:=\|f^2_i\|_{L^q(\mathbb{R}^n)}|2B_i|^{\frac{n}{q^{\prime}}}$
   and $a_i^2:=f_i^2/\lambda_i^2.$ Then we see that $a_i^2$ is a atom supported on $2B_i$ and
   $$|\lambda_i^2|\lesssim(|\tilde\lambda_i^1|+|\tilde\lambda_i^2|)|2B_i|\lesssim 1.$$

  Continuing in this process with $j\in{2,3,\cdots,J_0},$
  \begin{align*}
  	\tilde\lambda_i^j&:=\frac{1}{|2^jB_i|}\int_{2^jB_{i}}f_i(x)dx,\\
   	f_i^j&:=\tilde\lambda_i^{j-1}\chi_{2^{j-1}B_i}-\tilde\lambda_i^{j}\chi_{2^{j}B_i},\\
  	\lambda_i^j&:=\|f_i^j\|_{L^q(\mathbb{R}^n)}|2^{j}B_i|^{1/q^{\prime}},\\
  	a_i^j&:=f_i^j/\lambda_i^j,
  \end{align*}
 we obtain that
 $$f=\sum_{i=1}^{2}[\sum_{j=1}^{J_0}f_i^j]+\sum_{i=1}^{2}{\tilde\lambda_i^{J_0}}\chi_{2^{J_0}B_i}=\sum_{i=1}^{2}[\sum_{j=1}^{J_0}\lambda_i^ja_i^j]+
  \sum_{i=1}^{2}{\tilde\lambda_i^{J_0}}\chi_{2^{J_0}B_i},$$ where each $i$ and $j$, $a_i^j$ is a $(1,q,0)$-atom and $\lambda_i^j\lesssim 1.$

  For $\sum_{i=1}^{2}\tilde{\lambda}_i^{J_0}\chi_{2^{J_0}B_i},$ we set
   \begin{align*}
   \tilde{\lambda}^{J_0}&:=\frac{1}{|B(\frac{x_1+x_2}{2},2^{J_0+1}r)|}\int_{B(x_1,r)}f_1(x)dx\\
   &:=-\frac{1}{|B(\frac{x_1+x_2}{2},2^{J_0+1}r)|}\int_{B(x_2,r)}f_2(x)dx.
   \end{align*}
Which shows that
   \begin{align*}
   	&\sum_{i=1}^{2}\tilde{\lambda}_i^{J_0}\chi_{B(x_i,2^{J_0}r)}\\
   	&=[\tilde{\lambda}_1^{J_0}\chi_{B(x_1,2^{J_0}r)}-\tilde{\lambda}^{J_0}
   	\chi_{B(\frac{x_1+x_2}{2},2^{J_0+1}r)}]+[\tilde{\lambda}^{J_0}\chi_{B(\frac{x_1+x_2}{2},2^{J_0+1}r)}+\tilde{\lambda}_2^{J_0}\chi_{B(x_2,2^{J_0}r)}]\\
   	&=:\sum_{i=1}^{2}f_i^{J_0+1}.
   	\end{align*}
For $i=1,2$, let $$\lambda_i^{J_0+1}:=\|f_i^{J_0+1}\|_{L^{q}(\mathbb{R}^n)}|B(\frac{x_1+x_2}{2},2^{J_0+1}r)|$$ and $$a_i^{J_0+1}:=f_i^{J_0+1}/ \lambda_i^{J_0+1}.$$
Also, $a_i^{J_0+1}$ is a $(1,q,0)$-atom and $\lambda_i^{J_0+1}$ satisfies \eqref{lem1-2}. Thus, we have \eqref{lem1-1} holds, which implies
that $f\in H^1(\mathbb{R}^n)$ with
     $$\|f\|_{H^1(\mathbb{R}^n)}\le\sum_{i=1}^{2}\sum_{j=1}^{J_0+1}|\lambda_i^j|\lesssim \log\frac{|x_1-x_2|}{r}.$$
This finishes the proof of Lemma \ref{lem-main1}.
\end{proof}

     \begin{lemma} \label{lem-main2}
     Suppose  $1\le l\le m$, $1< p_{1},\dots,p_{m}<\infty$ and $0<\alpha<mn$ with
     $$\frac{1}{p_{1}}+ \dots +\frac{1}{p_{m}}-\frac{1}{q}=\frac{\alpha}{n}.$$
     There exists a positive constant $C$ such that for any $g\in L^{q^{\prime}}(\mu^{1-q^{\prime}}_{\vec{\omega}})$ and $h_{i}\in L^{p_{i}}(\omega_i^{p_i}), i=1,2,\dots ,m$,
     \begin{align*}
     \left\|\Pi_{l}\left(g,h_{1},\ldots,h_{m}\right)\right\|_{H^{1}\left(\mathbb{R}^{n}\right)}\leq{C}\|g\|_{L^{q^{\prime}}(\mu^{1-q^{\prime}}_{\vec{\omega}})}\left\|h_{1}\right\|_{L^{p_{1}
     }(\omega_1^{p_1})} \cdots\left\|h_{m}\right\|_{L^{p_{m}}(\omega_m^{p_m})}.
    \end{align*}
    \end{lemma}
\begin{proof}
 Note that for any $g\in L^{q^{\prime}}(\mu^{1-q^{\prime}}_{\vec{\omega}})$ and $h_{i}\in L^{p_{i}}(\omega_i^{p_i}), i=1,2,\dots ,m$, we have
\begin{eqnarray*}
\begin{aligned}
\int_{\mathbb{R}^n}|g(x)I_{\alpha}(h_{1},\cdots,h_{m})(x)|dx&=\int_{\mathbb{R}^n}|g(x)|\mu_{\vec{\omega}}(x)^{-\frac{1}{q}}\dot |I_{\alpha}(h_{1},\cdots,h_{m})(x)|\mu_{\vec{\omega}}(x)^{\frac{1}{q}}dx\\
&\leq \|g\|_{L^{q^{\prime}}(\mu^{1-q^{\prime}}_{\vec{\omega}})}\|I_{\alpha}(h_{1},\cdots,h_{m})\|_{L^{q}(\mu_{\vec{\omega}})}\\
&\leq C\|g\|_{L^{q^{\prime}}(\mu^{1-q^{\prime}}_{\vec{\omega}})}\prod_{i=1}^{m}\|h_{i}\|_{L^{p_{i}}(\omega_i^{p_i})}.
\end{aligned}
\end{eqnarray*}
On the other hand, the directly calculation gives us that
\begin{equation*}
\frac{1}{p'_{l}}=\sum_{j\neq l}\frac{1}{p_{j}}+\frac{1}{q'}-\frac{\alpha}{n} \qquad \text{and}\qquad
\omega_{l}^{-p_{l}}=\prod_{j\neq l}\omega_{j}^{p_{l}}\cdot \mu_{\omega}^{-p_{l}/q},
\end{equation*}
it follows that
$$I_{\alpha}: L^{p_{1}}(\omega_{1}^{p_{1}})\times\cdots\times L^{p_{l-1}}(\omega_{l-1}^{p_{l-1}})\times L^{q^{\prime}}(\mu^{1-q^{\prime}}_{\vec{\omega}})
\times L^{p_{l+1}}(\omega_{l+1}^{p_{l+1}})\times\cdots\times L^{p_{m}}(\omega_{m}^{p_{m}})\rightarrow L^{p'_{l}}(\omega^{-p_{l}}).$$
Which implies that ${\textstyle \prod_{l}}(g,h_1,\dots,h_m)(x)\in L^1(\mathbb{R}^n)$ by H\"{o}lder duality. Moreover,
$$\int_{\mathbb{R}^n} {\textstyle \prod_{l}}(g,h_1,\dots,h_m)(x)dx=0.$$
Hence, for $b\in BMO$, 		
   	\begin{equation*}
		\begin{split}
		\left|\int_{\mathbb{R}^n} b(x){\textstyle \prod_{l}}(g,h_1,\dots,h_m)(x)dx \right|
		&=\left|\int_{\mathbb{R}^n} g(x)\left [b,I_{\alpha} \right ]_l (h_1,\dots,h_m)(x)dx
		\right|\\
		&=\left|\int_{\mathbb{R}^n}g(x)\mu_{\vec{\omega}}(x)^{-\frac{1}{q}}\left[b,I_\alpha\right]_l(h_1,\dots,h_m)(x)\mu_{\vec{\omega}}(x)^{\frac{1}{q}}dx\right| \\
		&\le\left \|g\right\|_{L^{q^{\prime}}(\mu_{\vec{\omega}}^{1-q^{\prime}})}\cdot \left\|[b,I_\alpha]_l(h_1,\cdots,h_m)\right\|_{L^{q}(\mu_{\vec{\omega}})}\\
		&\le C\left \| h_1 \right \|_{L^{p_1}(\omega_1^{p_1})}\dots \left \| h_m \right \|_{L^{p_m}(\omega_m^{p_m})}\left \| g \right \|_{L^{q^{\prime}}(\mu^{1-q^{\prime}}_{\vec{\omega}})}\left \| b\right\|_{BMO}.
		\end{split}
      \end{equation*}
Therefore, ${\textstyle \prod_{l}}(g,h_1,\dots,h_m)$ is in $H^{1}(\mathbb{R}^n)$ with $$\left\|\Pi_{l}\left(g, h_{1}, \ldots, h_{m}\right)\right\|_{H^{1}\left(\mathbb{R}^{n}\right)} \leq C\|g\|_{L^{q^{\prime}}(\mu^{1-q^{\prime}}_{\vec{\omega}})}\left\|h_{1}\right\|_{L^{p_{1}}(\omega_1^{p_1})} \cdots\left\|h_{m}\right\|_{L^{p_{m}}(\omega_m^{p_m})}.$$
    The proof of Lemma \ref{lem-main2} is completed.
\end{proof}

     \begin{lemma} \label{lem-main3}
  Let $1\le l\le m$, $0<\alpha<mn$, $1< p_{1},\dots,p_{m},q<\infty$, $\frac{1}{p_1}+\dots+\frac{1}{p_m} -\frac{1}{q}=\frac{\alpha}{n}$ and $\vec{\omega}\in A_{\vec{P},q}$.
  For every $H^{1}(\mathbb{R}^n)$-atom $a(x)$,
     there exists $g\in L^{q^{\prime}}(\mu^{1-q^{\prime}}_{\vec{\omega}})$ and $h_{i}\in L^{p_{i}}(\omega_i^{p_i}), i=1,2,\dots ,m$ and a large positive number $M$(depending only on $\varepsilon$) such that:
     $$ \left \| a- {\textstyle \prod_{l}}(g,h_1,h_2,\dots ,h_m)\right \|_{H^1(\mathbb{R}^n)}<\varepsilon $$ and that
     $\|g\|_{L^{q^{\prime}}(\mu^{1-q^{\prime}}_{\vec{\omega}})}\left\|h_{1}\right\|_{L^{p_{1}}(\omega_1^{p_1})}\cdots\left\|h_{m}\right\|_{L^{p_{m}}(\omega_m^{p_m})}\le CM^{mn(1+L)-\alpha}$.
     \end{lemma}
\begin{proof}
Let $a(x)$ be an $H^{1}(\mathbb{R}^n)$-atom,supported in $B(x_0,r)$, satisfying that
    $$ \int_{\mathbb{R}^n}a(x)dx=0\qquad and \qquad \left \| a \right \|_{L^{\infty}(\mathbb{R}^n)}\le |B(x_{0},r)|^{-1}.$$
    Fix $1\le l \le m$. Now select $y_{l}\in \mathbb{R}^n$ so that $y_{l,i}-x_{0,i}=\frac{Mr}{\sqrt{n}}$, where $x_{0,i}$(reps.$y_{l,i}$) is the $i$-th coordinate of $x_{0}$(reps.$y_{l}$) for
    $i=1,2,\dots ,n$. Note that for this $y_{l}$, we have $\left | x_0-y_l \right | =Mr$. Similar to the relation of $x_0$ and $y_l$, we choose  $y_1$ such that $y_0$ and $y_1$
    satisfies the same relationship as $x_0$ and $y_l$ do.  Then by induction we choose $y_2,\dots ,y_{l-1},y_{l+1},\dots ,y_m$.
   We write $B_i=B(y_i,r)$, $\mu_i=\omega_i^{-p_i^{\prime}}\in A_{mp'_{i}}\subset A_{\infty}$ and set
 	\begin{equation*}
 	\begin{split}
 		&g(x):=\Big(\frac{|B_{l}|\mu_{\vec{\omega}}(x)}{\mu_{\vec{\omega}}(B_{l})}\Big)^{1/q}\chi_{B_l}(x),\\
        &g(B_{l}):=\int_{B(y_l,r)}g(z_{l})dz_{l},\\
 		&h_j(x):=\mu_{j}(x)\chi_{B_j}(x),\qquad j\ne l,\\
 		& h_l(x)=\frac{a(x)}{(I^*_{\alpha})_l(h_1,\dots,h_{l-1},g,h_{l+1},\dots,h_m)(x_0)}\frac{g(B_{l})\mu_{l}(x)}{\mu_{l}(B_{l})}\chi_{B_l}(x).
 	\end{split}
    \end{equation*}

It follows from the specific choice of the functions
    $h_1,\dots ,h_{l-1},g,h_{l+1},\dots,h_m$ that
    \begin{align*}
    	&|(I^*_{\alpha})_{l}(h_1,\dots,h_{l-1},g,h_{l+1},\dots,h_m)(x_0)|\\
    	&=\int_{B_1\times\cdots\times B_m}\frac{g(z_{l})\prod_{j\ne l}\mu_{j}(z_j)}{(|x_0-z_1|+\cdots+|x_0-z_m|)^{mn-\alpha}}dz_1\cdots dz_m\\
    	&\ge C(Mr)^{\alpha-mn}g(B_{l})\prod_{j\ne l}\mu_j(B_j).
    \end{align*}
The definitions of the functions $g(x)$ and $h_j(x)$ give us that $supp\ g=B(y_l,r)$ and $supp\ h_j=B(y_j,r)$. Moreover,
    \begin{align*}
    \|g\|_{L^{q^{}\prime}(\mu_{\vec{\omega}}^{1-q^{\prime}})}=\Big(\frac{|B_l|}{\mu_{\vec{\omega}}(B_l)}\Big)^{1/q}|B_l|^{1/q^{\prime}}=\frac{|B_l|}{\mu_{\vec{\omega}}(B_l)^{1/q}}
    \end{align*}
    and
    \begin{align*}
    \|h_j\|_{L^{p_j}(\omega_j^{p_j})}\lesssim \mu_j(B_j)^{\frac{1}{p_j}}
    \end{align*}
    for $i=1,\dots,l-1,l+1,\dots,m$. Also,
    \begin{align*}
    	\left\|h_{l}\right\|_{L^{p_{l}}(\omega_l^{p_l})}
    	& = \frac{1}{\left|(I^{*}_{\alpha})_{l}\left(h_{1}, \ldots, h_{l-1}, g, h_{l+1}, \ldots, h_{m}\right)\left(x_{0}\right)\right|}\frac{g(B_{l})}{\mu_{l}(B_{l})}\|a\|_{L^{p_{l}}(\mu_{l})} \\
    	&\lesssim Cr^{-n}(Mr)^{mn-\alpha}\mu_{l}(B_{l})^{-1/p'_{l}}\prod_{j\neq l}\mu_j(B_j)^{-1}. \tag{2.4}
    \end{align*}
Combining the estimates above, we arrive at
     \begin{align*}
    	\|g\|_{L^{q^{\prime}}(\mu^{1-q^{\prime}}_{\vec{\omega}})}\left\|h_{1}\right\|_{L^{p_{1}}(\omega_1^{p_1})}\cdots\left\|h_{m}\right\|_{L^{p_{m}(\omega_m^{p_m})}}
    	&\lesssim (Mr)^{mn-\alpha}\mu_{\vec{\omega}}(B_l)^{-1/q}\prod_{j=1}^{m}\mu_j(B_j)^{-1/p_j^{\prime}}.
    \end{align*}
From the inequality \eqref{weight-eq2}, we have
     $${\mu_j(B_j)}\lesssim {\mu_j((M+1)B_l)}\lesssim M^{nL}\mu_j(B_l).$$
From the fact that
     $$1\leq \left (\frac{1}{|B|}\int_{B}\prod_{i=1}^{m}\omega_i^{q}\right) \prod_{i=1}^{m} \left(\frac{1}{|B|}\int_{B}\omega_{i}^{-p_i^\prime}\right)^{\frac{q}{p_i^{\prime}}},$$
it is easy to see that
    $$\|g\|_{L^{q^{\prime}}(\mu^{1-q^{\prime}}_{\vec{\omega}})}\left\|h_{1}\right\|_{L^{p_{1}}(\omega_1^{p_1})}\cdots\left\|h_{m}\right\|_{L^{p_{m}(\omega_m^{p_m})}}
    \lesssim M^{mn(1+L)-\alpha}.$$
    Next, we have
	\begin{equation*}
	\begin{split}  
		&a(x)- {\textstyle \prod_{l}}(g,h_1,h_2,\dots ,h_m)(x)\\
		&=a(x)-\left [h_l(I^*_{\alpha})_l(h_1,\dots,h_{l-1},g,h_{l+1},\dots,h_m)-gI_{\alpha}(h_1,\dots,h_m)(x)\right ] \\
		&=a(x)\frac{(I^*_{\alpha})_l(h_1,\dots,h_{l-1},g,h_{l+1},\dots,h_m)(x_0)-(I^*_{\alpha})_l(h_1,\dots,h_{l-1},g,h_{l+1},\dots,h_m)(x)}{(I^*_{\alpha})_l(h_1,\dots,h_{l-1},g,h_{l+1},\dots,h_m)(x_0)} \\
		&\qquad+g(x){I_{\alpha}(h_1,\dots,h_m)(x})\\
		&=:W_1(x)+W_2(x)\\
	\end{split}
    \end{equation*}
It is obvious that $W_1(x)$ is supported on $B(x_0,r)$ and $W_2(x)$ is supported on $B(y_0,r)$.

 We first estimate $W_1(x)$. For $x\in B(x_0,r)$,
	we have
   \begin{align*}
   	&\left | W_1(x) \right | \\
   	&=\left|a(x) \right| \frac{\left|(I^*_{\alpha})_l(h_1,\dots,h_{l-1},g,h_{l+1},\dots,h_m)(x_0)-(I^*_{\alpha})_l(h_1,\dots,h_{l-1},g,h_{l+1},\dots,h_m)(x)\right|}
   	{\left|(I^*_{\alpha})_l(h_1,\dots,h_{l-1},g,h_{l+1},\dots,h_m)(x_0)\right|} \\
   	&\lesssim\frac{\|a\|_{L^{\infty}}}{(Mr)^{\alpha-mn}g(B_{l})\prod_{j\ne l}\mu_j(B_j)}\\
   &\qquad \times\int_{\prod_{j=1}^{m}B(y_j,r)}\frac{|x-x_0|g(z_l)\prod_{j\ne l}h_j(z_j)}{(\sum_{i=1,i\ne l}^{m}|z_l-z_i|+|z_l-x_0)^{mn-\alpha+1}}dz_1\cdots dz_m\\
   	&\lesssim\frac{r^{-n}}{(Mr)^{\alpha-mn}g(B_{l})\prod_{j\ne l}\mu_j(B_j)}
   \frac{rg(B_{l})\prod_{j\ne l}\mu_j(B_j)}{(Mr)^{mn-\alpha+1}}\\
   	&\lesssim \frac{1}{M r^n}.
   \end{align*}

    Next we estimate $W_2(x)$. From the definition of $g(x)$ and $h_i(x)$, we have
  	\begin{align*}
		&|I_{\alpha}(h_1,\cdots,h_m)(x)|\\
		&=\frac{1}{\left|(I^*_{\alpha})_l(h_1,\dots,h_{l-1},g,h_{l+1},\dots,h_m)(x_0)\right|} \\
		&\qquad \times\Big|\int_{\prod_{j\ne l}B(y_j,r)\times B(x_0,r)} \big(K(z_1,\dots,z_{l-1},x_0,z_{l+1},\dots,z_m)(x_0)\\
		&\qquad\qquad-K(z_1,\dots,z_{l-1},x,z_{l+1},\dots,z_m)(x_0)\big)a(z_l)\prod_{j\neq l}h_{j}(z_{j})dz_1\cdots dz_m\Big|\\
		&\lesssim\frac{\|a\|_{L^{\infty}}}{(Mr)^{\alpha-mn}\prod_{i=1}^{m}\mu_j(B_j)}\\
		&\qquad \times\int_{\prod_{j\ne l}B(y_j,r)\times B(x_0,r)}\frac{|x-x_0|\mu_{l}(z_{l})\prod_{j\ne l}|h_j(z_j)|}{(\sum_{i=1,i\ne l}^{m}|z_l-z_i|+|z_l-x_0)^{mn-\alpha+1}}dz_1\cdots dz_m\\
		&\lesssim \frac{1}{M r^n},
    \end{align*}
where in the second equality we use the cancelllation property of the atom $a(z_l)$. It follows that
	$$|W_2(x)|\lesssim\frac{g(x)}{M r^n}\chi_{B(y_l,r)}(x).$$
The estimates of $W_1(x)$ and $W_2(x)$ imply that
	\begin{align}\label{size}
	\left | a(x)- {\textstyle \prod_{l}(g,h_1,\dots ,h_m)(x)}\right|\lesssim \frac{1}{M r^n}\chi_{B(x_{0},r)}(x)+\frac{g(x)}{M r^n}\chi_{B(x_{l},r)}(x).
	\end{align}
Notice that
	\begin{align}\label{can}
	\int_{\mathbb{R}^n}\left [a(x)- {\textstyle \prod_{l}(g,h_1,\dots ,h_m)(x)} \right]dx=0,
	\end{align}
	because the atom $a(x)$ has cancellation property and the second integral equals 0 just by the definitions of $ {\textstyle \prod_{l}}$.
    Then the inequality \eqref{size} and the cancellation \eqref{can}, together with Lemma \ref{lem-main1}, show that
    $$ \left \| a(x)- {\textstyle \prod_{l}(g,h_1,\dots ,h_m)(x)} \right \|_{H^1(\mathbb{R}^n)}\le C\frac{\log M}{M}.$$
     For $M$ sufficiently large such that
     $$\frac{C\log M}{M}<\epsilon.$$
 Thus, the result follows from here.
\end{proof}

    With this approximation result above, we can give the proof of the main Theorem \ref{main1}.

    \vspace{0.5cm}

    \noindent
    \textbf{Proof of Theorem \ref{main1}.}  By Lemma \ref{lem-main2}, it is obvious that
    $$\left \|{\textstyle \prod_{l}(g,h_1,\dots ,h_m)(x)} \right \|_{H^1(\mathbb{R}^n)}\le C\left \| g \right \|_{L^{q^{\prime}}(\mu^{1-q^{\prime}}_{\vec{\omega}})}\left \|h_1\right \|_{L^{p_1}(\omega_1^{p_1})} \cdots \left \|h_m\right \|_{L^{p_m}(\omega_m^{p_m})}.$$
    It is immediate that for any representation of $f$ as in \eqref{f=}, i.e., $$f=\sum_{k=1}^{\infty}\sum_{s=1}^{\infty}\lambda_s^k{\textstyle\prod_{l}(g_s^k,h_{s,1}^k,\dots,h_{s,m}^k)(x)} ,$$
    with
         \begin{align*}
      \|f\|_{H^{1}\left(\mathbb{R}^{n}\right)}\leq C
      \inf\left\{\sum_{k=1}^{\infty}\sum_{s=1}^{\infty}\left|\lambda_{s}^{k}\right|
      \left\|g_{s}^{k}\right\|_{L^{q^{\prime}}(\mu^{1-q^{\prime}}_{\vec{\omega}})}
      \left\|h_{s,1}^{k}\right\|_{L^{p_{1}}(\omega_1^{p_1})}\cdots\left\|h_{s,m}^{k}\right\|_{L^{p_{m}}(\omega_m^{p_m})}\right\},
      \end{align*}
where the infimum above is taken over all possible representations of $f$ that satisfy \eqref{f=}.

    Next, we will show that the other inequality holds and that it is possible to obtain such a decomposition for any $f\in H^1(\mathbb{R}^n)$. Applying the atomic decomposition, for any $f\in H^1(\mathbb{R}^n)$ we can find a sequence $\left\{\lambda_s^1\right\}\in \ell^1$ and sequence of $H^1(\mathbb{R}^n)$-atom
    $\left \{a_{s}^{1} \right \}$ so that $f=\sum_{s=1}^{\infty}\lambda _s^1a_s^1$ and $\sum_{s=1}^{\infty}|\lambda_s^1|\le C\left \| f \right\|_{H^1(\mathbb{R}^n)} .$

   Fix $\varepsilon >0$ so that $C\varepsilon <1$. We apply Lemma \ref{lem-main3} to each atom $a_s^1$, then there exists $g_s^1\in L^{q^{\prime}}(\mu^{1-q^{\prime}}_{\vec{\omega}}),h_{s,1}^1\in L^{p_1}(\omega_1^{p_1}),\cdots ,h_{s,m}^1\in L^{p_m}(\omega_m^{p_m})$ with
    $$\left \| a_s^1- {\textstyle \prod_{j,l}}(g_s^1,h_{s,1}^1,\cdot,h_{s,m}^1)\right\|_{H^1({\mathbb{R}^n)}} <\varepsilon ,\qquad\forall s$$ and
    $$\left \| g_s^1 \right \|_{L^{q^\prime}(\mu^{1-q^{\prime}}_{\vec{\omega}})} \left \| h_1 \right \|_{L^{p_1}(\omega_1^{p_1})}\cdots \left \| h_m \right \|_{L^{p_m}(\omega_m^{p_m})}\le C(\varepsilon, L,\alpha),$$
    where $C(\varepsilon, L,\alpha)=CM^{mn(1+L)-\alpha}$ is a constant depending on $\varepsilon$, $L$ and $\alpha$. Notice that
   	\begin{equation*}
    \begin{split}
   	& f=\sum_{s=1}^{\infty}\lambda _s^1a_s^1={\textstyle \sum_{s=1}^{\infty}}\lambda_s^1{\textstyle\prod_{l}(g_s^1,h_{s,1}^1,\dots,h_{s,m}^1)(x)}
    	+{\textstyle \sum_{s=1}^{\infty}}\lambda_s^1(a_s^1-{\textstyle\prod_{l}(g_s^1,h_{s,1}^1,\dots,h_{s,m}^1)})\\
  	&=: M_1+E_1.
   	\end{split}
    \end{equation*}
 Moreover,
    $$\|E_1\|_{H^1(\mathbb{R}^n)}\le{\textstyle \sum_{s=1}^{\infty}}|\lambda_s^1|\ \|a_s^1-{\textstyle\prod_{l}(g_s^1,h_{s,1}^1,\dots,h_{s,m}^1)}\|_{H^1(\mathbb{R}^n)}\le \varepsilon {\textstyle \sum_{s=1}^{\infty}}|\lambda_s^1|\le \varepsilon C|f\|_{H^1(\mathbb{R}^n)}.$$

     In addition, since $E_1\in H^1(\mathbb{R}^n)$, we can also find a sequence $\left\{\lambda_s^2\right\}\in l^1$ and sequence of $H^1(\mathbb{R}^n)$-atom $\left \{a_{s}^{2} \right \}$ so that $E_1=\sum_{s=1}^{\infty}\lambda _s^2a_s^2$. and $$\sum_{s=1}^{\infty}|\lambda_s^2|\le C\left \| E_1 \right\|_{H^1(\mathbb{R}^n)} \le \varepsilon C^2\left \| f \right\|_{H^1(\mathbb{R}^n)}.$$

    Again, applying Lemma \ref{lem-main3} to each atom $a_s^2$, there exists $g_s^2\in L^{q^{\prime}}(\mu^{1-q^{\prime}}_{\vec{\omega}}),h_{s,1}^2\in L^{p_1}(\omega_1^{p_1}),\cdots ,h_{s,m}^2\in L^{p_m}(\omega_m^{p_m})$ with
    $$\left \| a_s^2- {\textstyle \prod_{j,l}}(g_s^2,h_{s,1}^2,\cdot,h_{s,m}^2)\right\|_{H^1({\mathbb{R}^n)}} <\varepsilon ,\qquad\forall s$$
    We then have that:
   	\begin{equation*}
   	\begin{split}
    & E_1=\sum_{s=1}^{\infty}\lambda _s^2a_s^2={\textstyle \sum_{s=1}^{\infty}}\lambda_s^2{\textstyle\prod_{l}(g_s^2,h_{s,1}^2,\dots,h_{s,m}^2)(x)}+{\textstyle \sum_{s=1}^{\infty}}\lambda_s^2(a_s^2-{\textstyle\prod_{l}(g_s^2,h_{s,1}^2,\dots,h_{s,m}^2)})\\
   	&=: M_2+E_2.
   	\end{split}
    \end{equation*}
	As before, observe that
	\begin{equation*}
	\begin{split}
	\|E_2\|_{H^1(\mathbb{R}^n)} &\le{\textstyle \sum_{s=1}^{\infty}}|\lambda_s^2|\ \|a_s^2-{\textstyle\prod_{l}(g_s^2,h_{s,1}^2,\dots,h_{s,m}^2)}\|_{H^1(\mathbb{R}^n)}\\
&\le \varepsilon{\textstyle\sum_{s=1}^{\infty}}|\lambda_s^2|
	\le (\varepsilon C)^2\|f\|_{H^1(\mathbb{R}^n)}.
	\end{split}
	\end{equation*}
   This gives us that
   	\begin{equation*}
    \begin{split}
   	& f=\sum_{s=1}^{\infty}\lambda _s^1a_s^1={\textstyle \sum_{s=1}^{\infty}}\lambda_s^1{\textstyle\prod_{l}(g_s^1,h_{s,1}^1,\dots,h_{s,m}^1)(x)}
    		+{\textstyle \sum_{s=1}^{\infty}}\lambda_s^1(a_s^1-{\textstyle\prod_{l}(g_s^1,h_{s,1}^1,\dots,h_{s,m}^1)})\\
   	&= M_1+E_1= M_1+M_2+E_2\\
    & =\sum_{k=1}^{2}\sum_{s=1}^{\infty }\lambda_s^k{\textstyle\prod_{l}(g_s^k,h_{s,1}^k,\dots,h_{s,m}^k)}+E_2.
   	\end{split}
    \end{equation*}

Continuing this process indefinitely, we obtain for each $1\le k\le K$ produces functions $g_s^k\in L^{q^{\prime}}(\mu^{1-q^{\prime}}_{\vec{\omega}}),h_{s,1}^k\in L^{p_1}(\omega_1^{p_1}),\cdots ,h_{s,m}^k\in L^{p_m}(\omega_m^{p_m})$ with $$\left \| g_s^k \right \|_{L^{q^\prime}(\mu_{\vec{\omega}})} \left \| h_{s,1}^k \right \|_{L^{p_1}(\omega_1^{p_1})}\cdots \left \| h_{s,m}^k \right \|_{L^{p_m}(\omega_m^{p_m})}\le C(\varepsilon, L,\alpha), \quad \forall s,$$ sequences $\left\{\lambda_s^k\right\}\in l^1$ with $\|\lambda_s^k\|_{l_1}\le \varepsilon^{k-1} C^k\left \| f \right\|_{H^1(\mathbb{R}^n)}$, and a function $E_K\in H^1(\mathbb{R}^n)$ with $$\|E_K\|_{H^1(\mathbb{R}^n)}\le (C\varepsilon)^K\left \| f \right\|_{H^1(\mathbb{R}^n)}.$$
The same argument above shows that
	$$f=\sum_{k=1}^{K}\sum_{s=1}^{\infty }\lambda_s^k{\textstyle\prod_{l}(g_s^k,h_{s,1}^k,\dots,h_{s,m}^k)}+E_k.$$
	Letting $K\to \infty $ gives the desired decomposition of $$f=\sum_{k=1}^{\infty}\sum_{s=1}^{\infty}\lambda_s^k{\textstyle\prod_{l}(g_s^k,h_{s,1}^k,\dots,h_{s,m}^k)}.$$ 	
We conclude that
$$\sum_{k=1}^{\infty}\sum_{s=1}^{\infty}|\lambda_s^k|\le \sum_{k=1}^{\infty}\varepsilon ^{-1}(C\varepsilon)^K\|f\|_{H^1(\mathbb{R}^n)}=\frac{C}{1-\varepsilon C}\|f\|_{H^1(\mathbb{R}^n)}.$$
Thus, we have completed the proof of Theorem \ref{main1}. \qed

\vspace{0.3cm}

	Finally, we dispense with the proof of Theorem \ref{main2}.
	
\vspace{0.5cm}

    \noindent
	\textbf{Proof of Theorem \ref{main2}.}  The upper bound in this theorem is contained in \cite{CW2013}. For the lower bound, suppose that $f\in{H^1(\mathbb{R}^n)}$, using the weak factorization in Theorem \ref{main1} and the weighted boundedness of $[b,I_{\alpha}]_l$, we obtain
	\begin{equation*}
	\begin{split}
	& \left \langle b,f \right \rangle_{L^2(\mathbb{R}^n)}=\sum_{k=1}^{\infty }\sum_{s=1}^{\infty }\lambda_s^k \left \langle b,{\textstyle\prod_{l}(g_s^k,h_{s,1}^k,\dots,h_{s,m}^k)} \right \rangle_{L^2(\mathbb{R}^n)}\\
	&\qquad \qquad=\sum_{k=1}^{\infty }\sum_{s=1}^{\infty }\lambda_s^k\left \langle g_s^k,{\left [ b,I_{\alpha}\right]_l(h_{s,1}^k,\dots,h_{s,m}^k)} \right \rangle_{L^2(\mathbb{R}^n)}\\
	\end{split}
	\end{equation*}\\
    Hence, we have that
  	\begin{equation*}
  	\begin{split}  
  	& |\left \langle b,f \right \rangle_{L^2(\mathbb{R}^n)}|\le \sum_{k=1}^{\infty }\sum_{s=1}^{\infty }|\lambda_s^k|\ \|g_s^k\|_{L^{q^{\prime}}(\mu^{1-q^{\prime}}_{\vec{\omega}})}
  		\|{\left [ b,I_{\alpha}\right]_l(h_{s,1}^k,\dots,h_{s,m}^k)}\|_{L^q(\mu_{\vec{\omega}})}\\
  	&\leq\left\|[b, I_\alpha]_{l}: L^{p_{1}}\left(\omega_1^{p_1}\right) \times \cdots \times L^{p_{m}}\left(\omega_m^{p_m}\right) \rightarrow L^{q}\left(\mu_{\vec{\omega}}\right)\right\| \\
  	&\qquad \times \sum_{k=1}^{\infty} \sum_{s=1}^{\infty}\left|\lambda_{s}^{k}\right|\left\|g_{s}^{k}\right\|_{L^{q^{\prime}}\left(\mu^{1-q^{\prime}}_{\vec{\omega}}\right)} \prod_{j=1}^{m}\left\|h_{s, j}^{k}\right\|_{L^{p_{j}}\left(\omega_j^{p_j}\right)}\\
  	& \leq C\left\|[b, I_\alpha]_{l}: L^{p_{1}}\left(\omega_1^{p_1}\right) \times \cdots \times L^{p_{m}}\left(\omega_m^{p_m}\right) \rightarrow L^{q}\left(\mu_{\vec{\omega}}\right)\right\|\ \|f\|_{H^1({\mathbb{R}^n})}. \\
   	\end{split}
   \end{equation*}
   From the duality theorem between $H^1(\mathbb{R}^n)$ and $BMO$, it follows that $b\in BMO$. \qed

\end{document}